\documentclass[12pt]{amsart}
\makeatletter
\def\@citecolor{blue}
\def\@urlcolor{blue}
\def\@linkcolor{blue}
\makeatother

\usepackage{hyperref}
\usepackage{xcolor}
\hypersetup{
    colorlinks,
    linkcolor={red!50!black},
    citecolor={blue!50!black},
    urlcolor={blue!80!black}
}
\usepackage[headings]{fullpage}
\usepackage{lmodern}
\usepackage[T1]{fontenc}
\usepackage{amsmath,amsthm,amssymb,amsfonts,latexsym,mathrsfs}
\usepackage[all]{xy}
\usepackage{paralist}
\usepackage{color}
\usepackage{ifthen}

\makeatletter

\@addtoreset{equation}{section}
\def\theequation{\thesection.\@arabic \c@equation}
\def\@citecolor{blue}
\def\@urlcolor{blue}
\def\@linkcolor{blue}
\def\theenumi{\@roman\c@enumi}

\makeatother

\theoremstyle{plain}
\newtheorem{theorem}[equation]{Theorem}
\newtheorem*{definition*}{Definition}
\newtheorem{lemma}[equation]{Lemma}
\newtheorem{corollary}[equation]{Corollary}
\newtheorem{proposition}[equation]{Proposition}
\newtheorem*{remark*}{Remark}
\newtheorem{question}[equation]{Question}
\newtheorem{conj}[equation]{Conjecture}
\theoremstyle{definition}
\newtheorem{remark}[equation]{Remark}

\newtheorem{remarks}[equation]{Remarks}

\newtheorem{example}[equation]{Example}

\newtheorem{definition}[equation]{Definition}

\def\NZQ{\mathbb}               
\def\NN{{\NZQ N}}

\def\AA{{\NZQ A}}

\def\ano=\mathcal{p}

\def\opn#1#2{\def#1{\operatorname{#2}}} 
\opn\supp{supp}
\opn\chara{char}
\opn\length{\ell}
\opn\projdim{proj\,dim}
\opn\depth{depth}
\opn\reg{reg}
\opn\lreg{lreg}
\opn\sat{^{sat}}
\opn\lex{^{lex}}
\opn\thh{\rm th}
\let\goth=\mathfrak
\opn\SK{Sk}
\opn\Char{char}
\opn\Ker{Ker}
\opn\Coker{Coker}
\opn\Im{Im}
\opn\Hom{Hom}
\opn\Tor{Tor}
\opn\Ext{Ext}
\opn\End{End}
\opn\Aut{Aut}
\opn\id{id}
\opn\GL{GL}
\opn\pf{\rm pf\,}
\opn\codim{codim}
\let\:=\colon
\let\ov\overline

\opn\pf{pf}
\opn\sgn{sgn}
\opn\Gin{Gin}
\opn\gin{gin}
\opn\Hilb{Hilb}
\opn\HilbPol{HilbPol}
\opn\ini{in}
\opn\End{end}

\begin{document}

\title{On Jet schemes of pfaffian ideals}
\author{Emanuela De Negri}
\address{Emanuela De Negri - Dipartimento di Matematica - Universit\`a di Genova - Via Dodecaneso 35 - 16146 Genova - Italy - OrcidID 0000-0001-5556-440X - Corresponding author}
\email{denegri@dima.unige.it}
\author{Enrico Sbarra}
\address{Enrico Sbarra - Dipartimento di Matematica - Universit\`a  di Pisa - Largo Bruno Pontecorvo 5 -  56127 Pisa - Italy - OrcidID 0000-0001-9656-6153 }
\email{enrico.sbarra@unipi.it}
\thanks{{\it Acknowledgments}: Our computations have been performed on a dedicated server, which has been acquired thanks to the support of the University of Pisa, within the call "Bando per il cofinanziamento dell'acquisto di medio/grandi attrezzature scientifiche 2016".
The authors were partially supported by INdAM-GNSAGA}
\begin{abstract}
Jet schemes and arc spaces received quite a lot of attention by researchers after their introduction, due to J.
Nash, and established their importance as an object of study in M. Kontsevich's motivic integration theory.
Several  results point out that jet schemes carry a rich amount of geometrical information about the original object they stem from, whereas, from an algebraic point of view, little is know about them. In this paper we study some algebraic properties of jet schemes ideals of pfaffian varieties and we determine under which conditions the corresponding jet scheme varieties are irreducible.  
\end{abstract}


\maketitle

\section*{Introduction}
The theory of jet schemes received a great impulse from the famous question known as the {\em Nash problem}, which was open for several decades, and was suggesting that certain geometrical properties of a scheme $X$ are somehow reflected in the geometry of its jet schemes $X_k$ and arc space $X_\infty$. A suggestive example of this idea is provided by the beautiful result proved by M. Musta\c{t}\u{a} \cite{Mu}:  if $X$ is locally a complete intersection variety over an algebraically closed field of characteristic 0, then $X_k$ is irreducible for all $k\in\NN_+$ if and only if $X$ has rational singularities. The interested reader will be able to find a rich bibliography (cf. \cite{Do} and \cite{Le} for a good introduction and set of references) which also explores the connections of this theory with that of $p$-adic and motivic integration, and with classical invariant theory - the presentation of which goes way beyond the scope of this note. 

\noindent
From an algebraic point of view, one can associate with any ideal $I$ in a polynomial ring a family of jet schemes ideals
$\mathcal{J}=\{I^{(k)}\}_{k\in\NN}$, and study, for instance, the numerical invariants of the elements of $\mathcal{J}$ and how they relate to those of the original ideal $I$. At present, such ideals are still quite mysterious objects, with some exceptions: the very special  case when $I$ is a monomial ideal is studied  in \cite{GoSm}, where sets of generators of the radical of $I^{(k)}$, for all $k\in\NN$, are determined; jet schemes ideals of the principal ideal generated by the squarefree product of all of the variables are described in \cite{Po}, from which it descends an improvement of the effective differential Nullstellensatz in Differential Algebra. B. A. Sethuraman and his collaborators dedicated quite a few articles to the topic: in \cite{NeSe}, \cite{SeSi} jet schemes of certain determinantal varieties appear as direct summands of the variety of commuting pairs in the centralizer of a 2-regular matrix; in \cite{KoSe} and, subsequently in \cite{Jo} and \cite{GhJoSe}, other properties of jet schemes of determinantal varieties are investigated, also from a computational and combinatorial point of view. 

The reader should keep in mind that, if the ambient space of the original ideal $I$ is $N$ dimensional, that of the $k$th jet scheme ideal $I^{(k)}$ of $I$ is $kN$ dimensional. Together with the fact that brute force Gr\"obner bases computations are very expensive, this fact explains why the calculations of determinantal examples are very limited \footnote{For instance, the computations we performed on the quotient ring of $I^{6,3}_2$ with \cite{Macaulay2} break up after a few minutes, whereas with \cite{CoCoA}  it exhausts the available 250Gb RAM memory in a few hours time.}.

In this note, following the approach of  \cite{KoSe}, we focus on jet schemes ideals $I^{n,k}_r$ of pfaffian ideals - i.e. ideals  generated by all of the pfaffians of size $2r$ of a skew-symmetric $n\times n$ matrix - which is a topic that is  is not treated in the literature yet. Observe that, determinantal varieties seldom verify the hypothesis of  Musta\c{t}\u{a}'s Theorem, e.g. pfaffian varieties are complete intersections only in the trivial cases of 2 pfaffians i.e. variables, and if $n$ is even and the pfaffian ideal $I$ is principal.
Specifically, in the first section we recall some basic facts about classic pfaffian ideals; in Section 2 we prove some results which yield a reduction process from the study of $I^{n,k}_r$ to that of some $I^{n_1,k_1}_{r_1}$, where, $n_1\leq n$,\, $k_1\leq k$,\, $r_1\leq r$. In the remaining sections we prove the main results of this paper, cf. Theorems \ref{main1}, \ref{main2}, \ref{reducibell} and Corollary \ref{perpiudi2}; in particular we establish in which cases the corresponding jet scheme varieties are irreducible.  

\section{Classical pfaffian Varieties}
Let $K$ be an algebraically closed field and let $X$  be an  $n\times n$ skew symmetric matrix of indeterminates over $K$, i.e. the entries $x_{ij}$ of $X$\, with $i<j$\, are indeterminates,\, $x_{ij}=-x_{ji}$\, for\, $i>j$,\, and\, $x_{ii}=0$ for $i=1,\ldots,n$.
 Let $K[X]=K[x_{ij}:1\le i<j\le n]$ be 
 the polynomial ring associated with the matrix $X$. Since a submatrix determined by rows and columns indexes $a_1,\ldots,a_{2r}$ of a given skew-symmetric matrix is  again skew-symmetric, one lets  $[a_1,\ldots,a_{2r}]$ denote the pfaffian of the submatrix $(x_{a_{i}a_{j}})$ of $X$, with\, $i,j=1,\ldots,2r\leq n$, which is called a {\em $2r$-pfaffian} or {\em a pfaffian of size $2r$}.
 Let  $I_r=I_{r}(X)$ be the ideal of  $K[X]$ generated by all the $2r$-pfaffians of $X$, which is the {\em classical}  pfaffian ideal and $\mathcal{P}_{r}$  the corresponding pfaffian variety in $ \mathbb{A}^{n(n-1)/2}$.  Its coordinate  ring $R_{r}(X)=K[X]/I_{r}(X)$ has been studied by many authors, cf. for instance \cite{AD}, \cite{Av}, \cite{De}, \cite{GhKr}, \cite{JoPr}, \cite{KlLa}, \cite{Ma}; we collect in the following proposition some of its most important properties. 

\begin{proposition}
\label{classpfaff} The ring $R_{r}(X)$ is a Gorenstein normal factorial domain of dimension
$(r-1)(2n-2r+1)$. The $a$-invariant of $R_{r}(X)$ is $(r-1)n$ and its multiplicity  is 
{\footnotesize $$e(R_{r}(X))=\det\left[{2n-4r+2 \choose  n-2r-i+j+1}- {2n-4r+2 \choose  n-2r-i-j+1}\right]_{i,j=1,\ldots,r-1}.$$} 
Its Hilbert series is given by the formula
{\footnotesize $$H(R_r, z)=\frac{z^{-{r-1\choose 2}}\det(a_{ij}(z))}
{(1-z)^d},$$}
where $d=\dim(R_{r}(X))$, and, for $1\le i,j\le r-1$,\,  $a_{ij}(z)$ is equal to
{\footnotesize $$\sum_{k\ge 0}\left[{ n-2r+j \choose k}{n-2r+i \choose k}-
  {n-2r+j+i-1 \choose k-1}{n-2r+1\choose k+1}\right]z^k.$$}
\end{proposition}

\noindent
The results contained in the above proposition show on the one hand that  pfaffian rings are endowed with many good properties, and on the other that its basic invariants can be computed despite their complexity.
We also recall that, by \cite{HeTr} and \cite{Ku}, the generators of $I_{r}(X)$ form a Gr\"obner basis with respect to any anti-diagonal order, and that similar results have been proven for larger classes of pfaffian ideals by the first author et al., see \cite{DeGo}, \cite{DeGo2}, \cite{DeSb}.

In the following proposition we recall a standard localization argument for pfaffian ideals (cf. \cite[Lemma 1.2]{JoPr}), an adaptation of which is needed for the proof of the Theorem \ref{redusion}, which follows from the fact that pfaffian ideals  are invariant under elementary operations on rows and columns which preserve skew-symmetry.

\begin{proposition}\label{classicloc}
  Let $X=(x_{ij})$, $Z=(z_{ij})$ be two skew-symmetric matrices of size $n\times n$ and $(n-2)\times (n-2)$ respectively, and let
 $\Psi$ be the set   
of the indeterminates contained in the $(n-1)${\rm th} and
$n${\rm th} rows and columns of $X$. 
Then the map\,  $\phi\: K[X][x_{n-1,n}^{-1}] \longrightarrow K[Z][\Psi][x_{n-1,n}^{-1}]$\, defined by
$$ \phi(x_{ij})=x_{ij} \text{\,\, if\, } x_{ij}\in\Psi\,\,\text{ and }\,\, \phi(x_{ij})=z_{ij}-\frac{x_{i,n}x_{j,n-1}-x_{i,n-1}x_{j,n}}{x_{n-1,n}}\,\, \text{ otherwise,} 
$$ is a ring isomorphism which yields that\,\,   $R_{r}(X)[x_{n-1,n}^{-1}]\simeq R_{r-1}(Z)[\Psi][x_{n-1,n}^{-1}].$
\end{proposition}

\section{Jet schemes and pfaffian Varieties}

Let $K$ be an algebraically closed field and let  $n, k$ be two fixed positive integers; let also $A^{n,k}=K[x_{ij}^{(h)} \: 1\leq i<j\leq  n,\;\;h=0,\ldots,k-1]$ be a polynomial ring over $K$, of dimension of $kn(n-1)/2$.
Now let $t$ be an indeterminate over $K$, and  
$X(t)=(x_{ij}(t))$ be the $n\times n$ skew-symmetric matrix with entries $x_{ij}(t)=x_{ij}^{(0)}+x_{ij}^{(1)} t+\ldots +x_{ij}^{(k-1)}t^{k-1}$,  for $i<j$.
Let $r\in\NN_+$, such that $2r \leq n$. The main object of our study are the so called {\em jet scheme ideals} (or {\em jet ideals} for short) $I^{n,k}_{r}\subseteq A^{n,k}$, and the quotient rings $R^{n,k}_r=A^{n,k}/I^{n,k}_{r}$, which we are going to define as  follows. Any pfaffian $p=p(t)$  of $X(t)$ can be written as $p=p^{(0)}+p^{(1)}t+\ldots+p^{(k-1)}t^{k-1}+p^{k}t^k+ \ldots +p^{(s)}t^{s}$, for some integer $s$ and $p^{(h)}\in A^{n,k}$,\, $h=1,\ldots,s$.\  The ideal $I^{n,k}_{r}$ is defined as the ideal generated by $p^{(h)}$, for all $h=0,\ldots,k-1$ and all $2r$-pfaffians $p$ of $X(t)$.

\begin{example}
Let $n=5,\, r=2,\, k=3$\, and\, $X(t)$ be the skew-symmetric matrix 
{\footnotesize $$X(t)=\begin{pmatrix} 0& x_{12}(t)& x_{13}(t)&x_{14}(t)&x_{15}(t)\\ \phantom{x_{12}(t)}&0&x_{23}(t)&x_{24}(t)&x_{25}(t)\\ \phantom{x_{12}(t)}& \phantom{x_{12}(t)}& 0& x_{34}(t)& x_{35}(t)\\ \phantom{x_{12}(t)}& \phantom{x_{12}(t)}& \phantom{x_{12}(t)}& 0& x_{45}(t)\\ \phantom{x_{12}(t)}&\phantom{x_{12}(t)}&\phantom{x_{12}(t)}&\phantom{x_{12}(t)}&0   
\end{pmatrix};$$}we are interested in ideals such as  $I_2^{5,3}=(p_i^{(h)} \: h=0,1,2,\,i=1,\ldots,5)$, where $p_i$ denotes here  the pfaffian of the matrix obtained from $X(t)$ by deleting its $i\thh$ row and column. For instance, {\footnotesize $p_5=x_{12}(t)x_{34}(t)-x_{13}(t)x_{24}(t)+x_{14}(t)x_{23}(t)$} gives rise to the three generators of $I_2^{5,3}$ 
{\footnotesize $$\begin{array}{c}p_5^{(0)}=x_{12}^{(0)}x_{34}^{(0)}-x_{13}^{(0)}x_{24}^{(0)}+x_{14}^{(0)}x_{23}^{(0)},\\[+.3cm]
  p_5^{(1)}=x_{12}^{(0)}x_{34}^{(1)}-x_{13}^{(0)}x_{24}^{(1)}+x_{14}^{(0)}x_{23}^{(1)}+x_{12}^{(1)}x_{34}^{(0)}-x_{13}^{(1)}x_{24}^{(0)}+x_{14}^{(1)}x_{23}^{(0)},\\[+.3cm]
  p_5^{(2)}=x_{12}^{(0)}x_{34}^{(2)}-x_{13}^{(0)}x_{24}^{(2)}+x_{14}^{(0)}x_{23}^{(2)}+x_{12}^{(1)}x_{34}^{(1)}-x_{13}^{(1)}x_{24}^{(1)}+x_{14}^{(1)}x_{23}^{(1)}+x_{12}^{(2)}x_{34}^{(0)}-x_{13}^{(2)}x_{24}^{(0)}+x_{14}^{(2)}x_{23}^{(0)},
  \end{array}$$
  }
\end{example}

\noindent
Clearly, when $k=1$, the ideal $I^{n,1}_{r}$ is the classical pfaffian ideal $I_{r}$ of a generic skew-symmetric matrix of size $n$ generated by all pfaffians of size $2r$ and $R_r^{n,1}\simeq R_r$ is the  quotient ring  which we described in the previous section.

\smallskip
The following lemma and theorem will be relevant for the induction arguments we are going to use later on. 

\begin{lemma}\label{zione}
Let $({\bf x}^{(0)})\subset K[X]$ denote the ideal generated by $x_{ij}^{(0)}$, for all\, $1\leq i <j \leq n$. 
  $$\begin{array}{cll}
    S^{n,k}_r/({\bf x}^{(0)}) & \simeq & \left\{\begin{array}{lc}
    R^{n,k-1}\ ,& \text{\,\, if\, } k\leq r;\\
    S^{n,k-r}_{r}[x_{ij}^{(h)}]_{h=k-r+1,\ldots,k-1}\ , & \text{ if\, } k>r.
    \end{array}\right.
    \end{array}
  $$
\end{lemma}

\begin{proof} 
  Let $p=[a_1,\ldots,a_{2r}]$ be a $2r$-pfaffian of $X(t)$. By one of the equivalent definition of pfaffian, we may write  $p=\sum_{\alpha\in\Pi}\sgn(\sigma_\alpha)x_{i_1,j_1}(t)\cdots x_{i_r,j_r}(t),$ where\,
      {\small $$\Pi:=\left\{ \{(i_1,j_1),\ldots,(i_r,j_r)\}\: \{i_1,\ldots,i_r,j_1,\ldots, j_r\}=\{a_1,\ldots,a_{2r}\},
      \begin{array}{l}i_1<i_2<\ldots<i_r,\,\text{ and }\\
     i_k<j_k\,\text{ for } k=1,\ldots,r 
      \end{array}\right\}$$ }
and, for $\alpha=\{(i_1,j_1),(i_2,j_2),\ldots,(i_r,j_r)\}$,\,
      $\sigma_{\alpha}$  is the permutation   
      {\footnotesize $\left [\begin{array}{cccccc} 1&2&3&4&\cdots&2r\\ i_1&j_1&i_2&j_2&\cdots&j_r
\end{array}\right ]$}.

\noindent      
Thus, $p=\sum_{h=0}^{s} p^{(h)}t^h$,  where $$p^{(h)}=\sum_{\alpha\in\Pi}\sgn(\sigma_\alpha)\sum_{l_1+\cdots +l_r=h}x^{(l_1)}_{i_1,j_1}\cdots x^{(l_r)}_{i_r,j_r}\ .$$ 
If $k\leq r$ then each $h\le k-1$ is strictly less than $r$ so that  all monomials appearing in each generator involve at least one  of the $x_{ij}^{(0)}$, so that $I^{n,k}_r\subseteq ({\bf x}^{(0)})$ and the first part of the statement is proven. Let  $k>r$. By the same argument, $(p^{(h)} \: 0\leq h\leq r-1)\subseteq ({\bf x}^{(0)})$. Now, if $r\leq h\leq k-1$, then 
$$\begin{array}{ll} p^{(h)}&=\sum\limits_{\alpha\in\Pi}\sgn(\sigma_\alpha)\left(\sum\limits_{\stackrel{l_1+\ldots+l_r=h}{l_i>0,\,\forall i}} x^{(l_1)}_{i_1,j_1}\cdots x^{(l_r)}_{i_r,j_r}+\sum\limits_{\stackrel{l_1+\ldots+l_r=h}{l_i>0,\,\exists i}} x^{(l_1)}_{i_1,j_1}\cdots x^{(l_r)}_{i_r,j_r}\right)\\
                          &=\sum\limits_{\alpha\in\Pi}\sgn(\sigma_\alpha)\left(\sum\limits_{\stackrel{l_1+\ldots+l_r=h}{l_i>0,\,\forall i}} x^{(l_1)}_{i_1,j_1}\cdots x^{(l_r)}_{i_r,j_r}\right)+g\ ,
\end{array}$$ where $g\in ({\bf x}^{(0)})$.
Now  we only need to observe that the first summand, after relabeling, is just a generator of $I^{n,k-r}_r$, and the proof is completed.
\end{proof}

\begin{remark}\label{modulotk}
In general, if $A$ is any ring, $f=f_0+f_1t+\cdots +f_st^s\in A[t]$ and $B=A[t]/(t^k)$,  then it is easy to see that $\ov{f}$ is invertible in $B$ if and only if $f_0$ is invertible in $A$. In our setting we consider jet ideals, which are ideals in the ring $A=A^{n,k}$, but we might as well consider them as ideals in $B=A[t]/(t^k)\simeq (K[t]/(t^k))[x_{ij}^{(h)} \: 1\leq i<j\leq  n,\;\;h=0,\ldots,k-1]$, where the above property comes at hand; in this case we say that we are {\em working modulo $t^k$}.
    
\end{remark}

\begin{theorem}\label{redusion}
Let $\Psi=\{x^{(h)}_{1n}, x^{(h)}_{2n},\ldots,x^{(h)}_{n-1,n}, x^{(h)}_{1,n-1}, x^{(h)}_{2,n-1},\ldots, x^{(h)}_{n-2,n-1} \: h=0,\ldots k-1\}$. Then  $$R^{n,k}_r[(x^{(0)}_{n-1,n})^{-1}]\simeq R^{n-2,k}_{r-1}[\Psi][(x^{(0)}_{n-1,n})^{-1}]\ .$$ Moreover, 
 there is a one-to-one correspondence between 
      $\{\goth p \in{\rm Min}(I^{n,k}_r) \: x^{(0)}_{n-1,n}\not\in \goth p\}$ and the set ${\rm Min}(I^{n-2,k}_{r-1})$ of minimal primes of $I^{n-2,k}_{r-1}$ which  preserves codimension.

\end{theorem}
\begin{proof}
Let $Z(t)=(z_{ij}(t))$ be the $(n-2)\times(n-2)$ skew-symmetric matrix whose entries are 
$z_{ij}(t)=z_{ij}^{(0)}+z_{ij}^{(1)} t+\ldots +z_{ij}^{(k-1)}t^{k-1}$, where $z_{ij}^{(h)}$ are indeterminates, and   consider the assignment 
$$\phi(x_{ij}(t))=x_{ij}(t) \text{\,\, if\, } x_{ij}(t)\in\Psi,\, \text { and } \phi(x_{ij}(t))=z_{ij}(t)-\frac{x_{i,n}(t)x_{j,n-1}(t)-x_{i,n-1}(t)x_{j,n}(t)}{x_{n-1,n}(t)},$$
otherwise.
Working modulo $t^k$, by Remark \ref{modulotk}, one has that $x_{n-1,n}(t)$ is invertible since $x^{(0)}_{n-1,n}$ is invertible. Therefore, by arguing as in the proof of \cite[Theorem 1.2]{KoSe}, we obtain the desired isomorphism and correspondence.
\end{proof}

\begin{corollary}\label{unoetutti}
Let $\goth p$ be a minimal prime of $I^{n,k}_r$; then, $x^{(0)}_{ij}\in \goth p$ for some $i$ and $j$ if and only if $x^{(0)}_{ij}\in \goth p$ for all $i$ and $j$. 
\end{corollary}

\begin{proof}\label{forus4} 
Without loss of generality, we only have to prove that, if $\goth p$ is a minimal prime of $I^{n,k}_r$ that does not contain  $x^{(0)}_{n-1,n}$, then $x^{(0)}_{ij}\not\in \goth p$ for any choice of $i<j$. The viceversa will then follow, by exchanging the roles of the indexes $n-1,n$ and $i,j$. Suppose by contradiction that there exist $i$ and $j$ such that $x_{ij}^{(0)}$ belongs to $\goth p$;  by the the previous theorem, $x_{ij}^{(0)}$ would be then mapped into the image  
$\goth q [\Psi][(x_{n-1,n}^{(0)})^{-1}]\subset R^{n-2,k}_{r-1}[\Psi][(x^{(0)}_{n-1,n})^{-1}]$ of $\goth p[(x_{n-1,n}^{(0)})^{-1}]$, where $\goth q\in{\rm Min}(I^{n-2,k}_{r-1})$. Since the image of $x_{ij}^{(0)}$ is $z_{ij}^{(0)}-(x_{in}^{(0)}x_{j,n-1}^{(0)}-x_{i,n-1}^{(0)}x_{jn}^{(0)})(x_{n-1,n}^{(0)})^{-1}$, it would follow that $z_{i,j}^{(0)}x_{n-1,n}^{(0)}-x_{i,n}^{(0)}x_{j,n-1}^{(0)}+x_{i,n-1}^{(0)}x_{j,n}^{(0)}\in \goth q[\Psi]$. Now, the coefficients of this expressions involve $\pm 1$, which would then belong to the prime ideal $\goth q$, a contradiction.

\end{proof}

 In the following we let $\mathcal{P}^{n,k}_{r}={\mathbb V}(I^{n,k}_{r})$ denote the pfaffian jet variety of $I^{n,k}_{r}$.   

\begin{remark}\label{sembrablabla}
  If we let $M=\{\goth p\in {\rm Min\,}(I_r^{n,k}) \,:\, x^{(0)}_{ij}\not\in \goth p \text{ for all } i, j\}$,\, $N=\{\goth p\in {\rm Min\,}(I_r^{n,k}) \,:\, x^{(0)}_{ij}\in \goth p \text{ for some } i, j\}$,\,$I=\bigcap\limits_{\goth p\in M} \goth p$\, and\, $J=\bigcap\limits_{\goth p\in N} \goth p$,\, then we may write
  $${\rm Min\,}(I_r^{n,k})=M\sqcup N,\,\,\,\,\,\,\text{and}\,\,\,\,\,\,\sqrt{I_r^{n,k}} =I\cap J.$$
  Let, as before, $({\bf x}^{(0)})$ be the ideal generated by all of the $x^{(0)}_{ij}$; moreover let \,\, $\mathcal{Z}={\mathbb  V}(I)$\,\, and\,\, $\mathcal{Y}=\mathcal{P}_r^{n,k}\cap {\mathbb  V}(({\bf x}^{(0)}))$. Clearly, \, $\mathcal{P}_r^{n,k}={\mathbb V}(I)\cup{\mathbb V}(J)=\mathcal{Y}\cup \mathcal{Z}$.\, Notice that $\mathcal{Z}$ is not empty and\, $\mathcal{Z}\not\subseteq\mathcal{Y}$;\, moreover, $\mathcal{Z}$ is the union of its irreducible components which correspond to the set of minimal primes defining $I$, i.e. those not containing $x_{n-1,n}^{(0)}$. By Theorem \ref{redusion}, there is a bijection between these  and the irreducible components of $\mathcal{P}^{n-2,k}_{r-1}$, which preserves codimensions.  Finally, by Lemma \ref{zione}, $\mathcal{Y}$ is isomorphic to either $\AA^{n(n-1)(k-1)/2}$ when $k\leq r$ or to $\mathcal{P}_r^{n,k-r}\times\AA^{n(n-1)(r-1)/2}$ when $k>r$.
\end{remark}

\section{The irreducible cases}
Philosophically speaking, for classical determinantal and pfaffian varieties, the so called maximal cases are usually endowed with better properties. In this section and in the next one we shall focus on the problem of irreducibility of pfaffian jet varieties (or {\em pfaffian jet varieties} for short) and it will turn  out that this general behaviour is confirmed.

In this section we shall consider the special cases when $2r=n$ and $2r+1=n$, and prove that the corresponding pfaffian jet varieties are irreducible.

\medskip
\noindent
\framebox{$n=2r$.} This is the easiest case: the ideal $I^{2r,k}_r=I^{n,k}_{n/2}$ has  exactly $k$ generators, which are  $p^{(h)}$, $h=0,\ldots,k-1$, where $p=[1,\ldots,n]$, the pfaffian of the whole matrix $X(t)$. We consider now the degrevlex order $\prec$ on the monomials of $A^{n,k}$ induced by  {\footnotesize $x_{1,2}^{(k-1)}>\ldots >x_{1,n}^{(k-1)}>x_{2,3}^{(k-1)}>\ldots >x_{n-1,n}^{(k-1)}>x_{1,2}^{(k-2)}>\ldots >x_{1,n}^{(k-2)}>x_{2,3}^{(k-2))}>\ldots >x_{n-1,n}^{(k-2)}>\ldots>x_{1,2}^{(0)}>\ldots>x_{1,n}^{(0)}>x_{2,3}^{(0)}>\ldots>x_{n-1,n}^{(0)}$}.

\begin{theorem}\label{main1} Let $n=2r$. Then,
\begin{itemize}
\item[a)] the Gr\"obner basis w.r.t. to $\prec$  of the pfaffian jet ideal $I^{2r,k}_r$ is given by $$[1,\ldots,n]^{(0)},\ldots,[1,\ldots,n]^{(k-1)};$$
\item[b)] the jet pfaffian ideal $I^{2r,k}_r$ is prime of codimension $k$;\\
\item[c)] $R^{2r,k}_r$ is a complete intersection and, therefore, Cohen-Macaulay  of multiplicity $r^k$.
\end{itemize}
\end{theorem}
\begin{proof}
Let $p=[1,\ldots,n]$. For the proof of {\em a})\, it is sufficient to observe that the leading term  of $p^{(h)}=\sum_{\alpha\in\Pi}\sgn(\sigma_\alpha)\sum_{l_1+\cdots +l_r=h}x^{(l_1)}_{i_1,j_1}\cdots x^{(l_r)}_{i_r,j_r}$ is given by 
$$\ini_\prec(p^{(h)})=x_{1,2r-2t}^{(q)}\cdots x_{r-t-1,r-t+2}^{(q)}\cdot x_{r-t,r-t+1}^{(q)} \cdot x_{2r-2t+1,2r}^{(q+1)}\cdot x_{2r-2t+2,2r-1}^{(q+1)}\cdots x_{2r-t,2r-t+1}^{(q+1)}\,\,,$$ where $h=qr+t$, with $0\leq t\leq r-1$. Therefore, it is easy to see that any two distinct  such monomials are relatively prime, and it is well known that in
this situation one has a Gr\"obner basis.

\noindent
{\em b})\,: We shall prove that $R^{2r,k}_r$ is a domain of dimension $d:=k(n-2)(n+1)/2$ by induction on $r$. If $r=1$ and $n=2$, then $I^{2,k}_1=(x_{12}^h \: h=0,\ldots,k-1)$ and $R^{2,k}_1\simeq K$ is a $0$-dimensional domain. We now let $r>1$ and assume that $R^{n-2,k}_{r-1}$ is a domain of dimension $k(n-4)(n-1)/2$; let also $\Psi$ be as in Theorem \ref{redusion}. Then, $R^{n-2,k}_{r-1}[\Psi]$ has dimension $k(n-4)(n-1)/2+k(n-1+n-2)=d$, which is also the dimension of $R^{n-2,k}_{r-1}[\Psi][(x^{(0)}_{n-1,n})^{-1}]$. Theorem \ref{redusion} now yields that $R^{n,k}_r[(x^{(0)}_{n-1,n})^{-1}]$  is a domain and so is $R^{n,k}_r$: in fact, by Part {\em a}), $x_{n-1,n}^{(0)}$ is regular modulo $\ini_\prec I^{2r,k}_r$ and, by a well known property of the reverse lexicographic order, it is also $R^{2r,k}_r$-regular. Finally, $\dim R^{n,k}_r=\dim R^{n,k}_r[(x^{(0)}_{n-1,n})^{-1}]=d$, as desired. 
\noindent
Part {\em c}) follows now easily from the above.
\end{proof}

\begin{remark}\label{aldo} If $n=2r$, the $k$th jet scheme is generated by a complete intersection of $k$ elements of degree $r$, for any $k$; therefore we also know the Hilbert series of $R^{n, k}_r$, which is
$$
  H(R^{n,k}, z) =  \frac{\prod_{j=1}^k \sum_{i=1}^{r-1}z^i}{(1-z)^{\dim A^{n,k}}}=\left(\frac{\sum_{i=1}^{r-1}z^i}{(1-z)^{\dim A^{n,1}}}\right)^k.$$
  Thus,
  \begin{equation}\label{oone}
    H(R^{n,k}, z) =  H(R^{n,1}, z)^k  \text{\,\,\, and, consequently\,}
  \end{equation}
  \begin{equation}\label{ttwo}
    \dim R^{n,k}=k\dim R^{n,1},\,\,\,\,e(R^{n,k})=e(R^{n,1})^k
    .
  \end{equation}
\end{remark}

\medskip
\noindent
\framebox{$n=2r+1$.} We recall first some general basic facts and set some notation useful henceforth.

Let $M$ be any $m \times m$\, skew-symmetric matrix. 
 For for $l=1,\ldots, m$  we denote by $M_l$ the $(m-1) \times (m-1)$ skew-symmetric matrix obtained by removing the $l$th row and column from $M$. Accordingly, we denote by  $\pf_l(M)$ the pfaffian of $M_l$. 
Let  $\lambda\in K$; if $\lambda\neq 0$, we let $E_i(\lambda)$ be the elementary row multiplication matrix and, for any $\lambda$, we let  $E_{ij}(\lambda)$ be an elementary row addition matrix. If we let\, $M'=E_i(\lambda)\, M\, E_i(\lambda)$\, and\, $M''=E_{ij}(\lambda)\, M\, E_{ij}(\lambda)^{\rm t}$, then it is well-known that
${\rm pf\,} M' = |\lambda|{\rm pf\,} M$\, and\,  ${\rm pf\,} M'' = {\rm pf\,} M$.

Let now $M$ be a skew-symmetric matrix of size $2r+1$. We associate a point $\mathscr{ P}\in {\mathbb A}_K^{(2r+1)rk}$ of coordinates $p_{ij}^{(h)}$, where $1\leq i < j \leq 2r+1$,\, $0\leq h \leq k-1$, with the  matrix $A_\mathscr{ P}=(p_{ij}(t))$, where 
$p_{ij}(t)=p_{ij}^{(0)}+p_{ij}^{(1)}t+ \cdots + p_{ij}^{(k-1)}t^{k-1}$, and viceversa. Clearly, $\mathscr{P}\in \mathcal{P}_r^{2r+1,k}$  if and only if the first $k$ coefficients $\pf_l(A_\mathscr{P})^{(0)},\ldots, \pf_l(A_\mathscr{P})^{(k-1)}$\, of\, $\pf_l(A_\mathscr{P})$\, are zero for all  $l=1,\ldots, 2r+1$. 

\begin{lemma}\label{smanettoso} 
  Let  $\mathcal{Y}$, $\mathcal{Z}\subset \mathcal{P}_r^{n,k}$  be as in Remark \ref{sembrablabla},  and for any\, $i, j$,\, let $\mathcal{U}_{ij}$ be the open set of $\mathcal{P}_r^{n,k}$ where $x_{ij}^{(0)}\not=0$. Then one has:
\begin{itemize}
\item[a)] $\mathcal{Z}= \ov{\mathcal{U}_{ij}}$,\, where $\ov{\phantom{\mathcal{U}}}$ denotes the Zariski closure;
\item[b)]  if $n=2r+1$, then $\mathcal{Y} \subseteq \mathcal{Z}$. 
\end{itemize}
\end{lemma}
\begin{proof}  
Part {\em a}) can be proved along the lines of the proof of  \cite[Lemma  2.7]{KoSe}. For {\it b}), we prove that for all $\mathscr{Q} \in \mathcal{Y}$ there exist $i, j$, with\, $i < j$\, such that $\mathscr{Q}\in \ov{\mathcal{U}_{ij}}$; in this way the conclusion will be yielded  by Part {\em a}).

Let\, $A_\mathscr{Q}=(q_{ij}(t))$,\, with\,  $q_{ij}(t)=q_{i, j}^{(0)}+q_{i, j}^{(1)}t+ \cdots + q_{i, j}^{(k-1)}t^{k-1}$,\, be the matrix associated with $\mathscr{Q}$.
Recall that $\mathscr{Q}$ belongs to $\mathcal{Y}$ exactly when it belongs to $\mathcal{P}_r^{2r+1,k}$ and $q_{ij}^{(0)}=0$ for all $i, j$.
If $\mathscr{Q}$ is the origin, we choose any $i <j$ and a point $\mathscr{P}\in \mathcal{U}_{ij}$, since the latter is not empty, and consider the family of points $\mathscr{P}_\lambda=\lambda \mathscr{P}$, with $\lambda\in K$. Clearly, $\mathscr{P}_0=\mathscr{Q}$\, and $\mathscr{P}_\lambda\in \mathcal{U}_{ij}$ for all $\lambda\neq 0$; therefore, by the properties of   Zariski topology, $\mathscr{Q}\in\ov{\mathcal{U}_{ij}}$.

If  $\mathscr{Q}$ is not the origin, consider\, $i_0, j_0$\, such that\, $q_{i_0, j_0}(t)\neq 0$; by exchanging rows and columns if necessary, we may  assume without loss of generality that\, $i_0\not=2r+1$\, and\, $j_0>i_0+1$.\, Let  $s >0$ be the minimum such that $q_{i_0, j_0}^{(s)}\neq 0$. We shall show that $\mathscr{Q}\in \ov{\mathcal{U}_{i_0+1, j_0}}$. 
Let $E=E_{i_0+1, i_0}(\lambda/t^s)$ be an elementary row addition  matrix; starting with $A_\mathscr{Q}$ we construct a family of matrices $A_{\mathscr{P}_\lambda}=E \, A_\mathscr{Q}\, E^{\rm t}$, depending on a parameter $\lambda \in K$ , which corresponds to a family of points $\mathscr{P}_{\lambda}$. Clearly, $\mathscr{P}_0=\mathscr{Q}$. Moreover, for all $\lambda\neq 0$, we have that $\mathscr{P}_\lambda\in \mathcal{U}_{i_0+1, j_0}$, since $(p_\lambda)_{i_0+1, j_0}^{(0)} = q_{i_0+1, j_0}^{(0)} + \lambda q_{i_0, j_0}^{(s)}=\lambda q_{i_0, j_0}^{(s)}\neq 0$. 
In order to conclude the proof now we only need to show that $\mathscr{P}_\lambda \in \mathcal{P}_r^{2r+1,k}$ for all $\lambda\neq 0$, i.e. that,\, $\pf_l(A_{\mathscr{P}_\lambda})^{(h)}=0$\, for\, $h=0,\ldots,k-1$\, and\,  $l=1,\ldots,2r+1$. 
If $l=i_0+1$, we remove the column and row we operated on and  we obtain $(A_{\mathscr{P}_\lambda})_l=(A_\mathscr{Q})_l$ and, therefore, $\pf_l(A_{\mathscr{P}_\lambda})=\pf_l(A_\mathscr{Q})$. 
Let now\, $l=i_0$,\, and  $B$ be the matrix obtained by $A_\mathscr{Q}$ by removing the $(i_0+1)$th row and column and multiplying the $i_0$th row and column by a non-zero factor $\lambda/t^s$;\, in this case  we have that\, $\pf_l\left(A_{\mathscr{P}_\lambda}\right)\, =\, \pf_l\left(A_\mathscr{Q}\right) + \pf( B)\, =\, \pf_l\left(A_\mathscr{Q}\right) +\lambda/t^s\pf_{l+1}\left(A_\mathscr{Q}\right)$.  Finally, if $l\neq i_0,\,i_0+1$ then $(A_{\mathscr{P}_\lambda})_l= (E\, A_\mathscr{Q}\, E^{\rm t})_l= E_l\, (A_\mathscr{Q})_l\, (E_l)^{\rm t}$ and, thus,\,   $\pf_l(A_{\mathscr{P}_\lambda})=\pf_l(A_\mathscr{Q})$. Now, working modulo $t^k$, since by hypothesis  $\pf_l(A_\mathscr{Q})^{(h)}=0$\, for\, every $l$ and for  $h=0,\ldots,k-1$, the conclusion is straightforward.

\end{proof}

The next result generalizes what it happens in the classical case: if $k=1$ and $n=2r+1$, then the pfaffian ideal has codimension $3$.

\noindent
\begin{theorem}\label{main2} 
  Any pfaffian jet  variety $\mathcal{P}^{2r+1,k}_r$ is irreducible of codimension $3k$.
\end{theorem}

\begin{proof}
The proof runs by induction on $r$. When $r=1$ and $n=3$, the ideal $I^{3,k}_1$ is generated by $\{x^{(h)}_{12}, x^{(h)}_{13}, x^{(h)}_{23} \: \,h=0,\ldots, k-1 \}$ and $\mathcal{P}^{3,k}_1=\{O\}\subset \AA^{3k}$, which is irreducible of the right codimension.  Let us now suppose that $r>1$ and that the varieties $\mathcal{P}^{2s-1,k}_{s-1}$, for $s\leq r$ and all $k$, are irreducible of codimension 3k. Thus, the radical of $I^{2r-1,k}_{r-1}$ is prime and, by Theorem \ref{redusion}, also the radical of $I^{2r+1,k}_r$ is a prime ideal, that does not contain $x_{n-1,n}^{(0)}$. By Remark \ref{sembrablabla}, it is now sufficient to prove that $\mathcal{Y}\subseteq \mathcal{Z}$, but that is yielded by Lemma \ref{smanettoso}.
\end{proof}

\begin{example}
  Consider $I^{5,3}_2$, the jet  pfaffian ideal of the ideal generated by the 4 pfaffians of a generic $5\times 5$ skew symmetric matrix $X(t)$ of polynomials of degree 2 in $t$;\, $I^{5,3}_2$ is generated by 15 polynomials of degree r=2. With Macaulay2 we can verify that $\codim I^{5,3}_2=9$, as implied by the previous result, and that $I^{5,3}_2$ is prime.  Moreover, the Hilbert series of $R^{5,3}_2$ is
 {\footnotesize $$H(R^{5,3}_2, z)=\frac{1+9z+30z^2+45z^3+30z^4+9z^5+z^6}{(1-z)^{21}},$$}
  and {\footnotesize $$H(R^{5,1}_2, z)=\frac{1+3z+z^2}{(1-z)^7}.$$}
  Thus,\,\, $H(R^{5,3}_2, z)=(H(R^{5,1}_2, z))^3$.

  \noindent
Some computational evidence suggests that the given set of generators is not a Gr\"obner basis with respect to monomial orders which are in a natural way mutated from those which are classically considered in the literature for ideals of Pfaffians. 
\end{example}

The corresponding classical pfaffian ideals in this case  are even prime; we expect the same kind of behaviour for the general jet  pfaffian ideal $I^{2r+1,k}_r$.

\begin{conj}\label{quno}
For all positive integers $r$ and $k$,\, the jet pfaffian ideals  $I^{2r+1,k}_r$ are prime.
\end{conj}

\noindent

Considering the size of the computations involved, an affirmative answer to the following would provide a great simplification for the calculation of the Hilbert series of such rings.

\begin{question}\label{qdue}
  Do Formulas \eqref{oone} and \eqref{ttwo} hold true for $R^{2r+1,k}_r$, for all positive integers $r$ and $k$?
\end{question}

\begin{question}\label{cm}
Are $\mathcal{P}^{2r+1,k}_r$ Cohen Macaulay, for all positive integers $r$ and $k$?
\end{question}

\section{The reducible cases}

We now want to prove that the pfaffian jet  ideals are reducible except for the cases considered in the previous section.

\noindent
\framebox{$n\geq 2r+2$} The main result of the section is the following theorem.
 \begin{theorem}
\label{reducibell}
The pfaffian jet variety $\mathcal{P}^{n,k}_{r}$ is reducible 
for every $n\geq 2r+2$, with $r\geq 2$. Moreover, $\mathcal{P}^{n,k}_{r}$  is not pure; in particular, it is not Cohen Macaulay. 
\end{theorem}

 In order to prove the theorem, we want to use the reduction argument of Lemma \ref{redusion}; therefore we first focus on  the case of jet varieties of $4$-pfaffians, with $n\ge 6$.

\noindent
We start by describing a procedure to obtain the irreducible components of $\mathcal{P}^{n,k}_2$, with $n\geq 6$. 

Recall that, as discussed in Remark \ref{sembrablabla},\, we may write $\mathcal{P}^{n,k}_{2}=\mathcal{Y}_0\cup \mathcal{Z}_0$,\, where $\mathcal{Y}_0$ is the subvariety where $x_{ij}^{(0)}=0$ for all $i,j$, and $\mathcal{Z}_0$ is the closure of any of the open sets $\mathcal{U}_{ij}$ - see Lemma \ref{smanettoso}  {\em a}). In contrast to what happens in the maximal case, cf. Lemma \ref{smanettoso} {\em b}), we prove  next that there is not an inclusion between the two subvarieties.

\begin{lemma}\label{new}
 Let $r=2$, $n\ge 6$, and\, $\mathcal{P}^{n,k}_{2}=\mathcal{Y}_0\cup \mathcal{Z}_0$.\, Then, one has $\mathcal{Y}_0\not\subseteq \mathcal{Z}_0$.
\end{lemma}
\begin{proof} Let $\mathscr{Q}$ be a point which is in $\mathcal{U}_{ij}\subset \mathcal{Z}_0$ for some $i,j$, and 
let\, $A_\mathscr{Q}=(q_{ij}(t))$,\, with\,  $q_{ij}(t)=q_{i, j}^{(0)}+q_{i, j}^{(1)}t+ \cdots + q_{i, j}^{(k-1)}t^{k-1}$,\, be the matrix associated with $\mathscr{Q}$.
Since  $\mathscr{Q}$ belongs to $\mathcal{P}^{n,k}_{2}$, for every  $4$-pfaffian $p$ of $A_\mathscr{Q}$ one has $p^{(h)}=0$, for all $h=0,\ldots,k-1$.
Let $q$ be any $6$ pfaffian of $A_\mathscr{Q}$. For what we are going to say next, we may assume without loss of generality  that $Q$ is the 
submatrix of  the first $6$ rows and columns of $A_\mathscr{Q}$, and that $q=\pf(Q)$.
Since by Laplacian expansion  $q$ can be expressed in terms of $4$-pfaffians, we have that $q^{(h)}=0$, for $h=0,\ldots,k-1$ and, therefore, $q=0$ modulo $t^k$. Thus also $\det Q=0$ modulo $t^k$. Again without loss of generality, we may assume $(i,j)=(5,6)$ and that, working modulo $t^k$, we may express the last column of $Q$ as a linear combination of the previous ones. Thus, there exist
  polynomials $\lambda_h(t)$, $h=1,\ldots,5$ of degree at most $k-1$ such that $q_{i6}(t)=\sum_{h=1}^5 \lambda_h(t)q_{ih}(t) \mod t^k$ for $i=1,\ldots,5$, i.e. there exists polynomials $\mu_i$ such that  $q_{i6}(t)=\sum_{h=1}^5 \lambda_h(t)q_{ih}(t) + t^k\mu_i$ for $i=1,\ldots,5$. Consider now the $6\times 6$ skew-symmetric matrices
  $Y=(y_{ij})$, $Z=(z_{ij})$, defined by
{\footnotesize  $$y_{ij}=\left\{\begin{array}{ll} q_{ij}(t) &\text{ if }   1\leq i <   j \leq 5,\\[+.3cm]

  \smallskip
  \sum_{h=1}^5\lambda_h(t)q_{ih}(t) & \text{ if }
  1\leq i < j=6;
  \end{array}
  \right.
  \,\,\,\,\,
  z_{ij}=\left\{\begin{array}{ll} q_{ij}(t) &\text{ if }   1\leq i < j \leq 5,\\[+.3cm]
  t^k\mu_i & \text{ if } 1\leq i < j=6.
  \end{array}
  \right.
  $$
  }
  Thus,\, $\pf (Y)=0$\,  and\,
  \begin{equation}\label{crux}
    q=\pf (X)=\pf (Y) + \pf (Z)= \sum_{h=1}^5 \pm t^k\mu_h \pf_{h6} Z = t^k\sum_{h=1}^5 \pm \mu_h \pf_{h6} Q.
  \end{equation}
  From the above, we derive that a point belongs to the (non-empty) open subset $\mathcal{U}_{56}\subset \mathcal{Z}$  verifies the extra conditions $$q^{(h)}=0\,\,\,\, \text{ for }\,\,\,\, h=k,\ldots 2k-1.$$ On the other hand, consider the point $\mathscr{P}$ determined by the associated matrix
  \begin{equation}\label{crux2}
    {\footnotesize A_\mathscr{P}=\begin{pmatrix} 0& t^\ell&0 &0 &0 &0 \\
 &0 &0 &0 &0 &0 \\
 & &0 &t^{\ell+1} &0 &0 \\
 & & &0 &0 &0 \\
 & & & &0 & t^{k-1}\\
 & & & & &0 \\
 \end{pmatrix}},\text{ if } k=2\ell+1\,\,\,\text{ or }\,\,\,
{\footnotesize A_\mathscr{P}=\begin{pmatrix} 0& t^\ell&0 &0 &0 &0 \\
 &0 &0 &0 &0 &0 \\
 & &0 &t^\ell &0 &0 \\
 & & &0 &0 &0 \\
 & & & &0 & t^{k-1}\\
 & & & & &0 \\
 \end{pmatrix}},
  \end{equation}

\noindent if $k=2l$; in both cases $\pf(A_\mathscr{P})=t^{2k-1}$, and it is not difficult to see that   $\mathscr{P}\in \mathcal{Y}_0\setminus \mathcal{Z}_0$\,  for\, $s=0,\ldots,\ell-1$.
 \end{proof}

\begin{proposition}\label{4reducible}
  The jet  variety $\mathcal{P}^{n,k}_2$ of a pfaffian ideal generated by $4$-pfaffians is reducible for all $n\geq 6$.

\end{proposition}
\begin{proof} 
The previous lemma implies that  $\mathcal{P}^{n,k}_2=\mathcal{Y}_0\cup\mathcal{Z}_0$ yields a decomposition of our variety as a union of two proper subvarieties.
\end{proof}

Now we go on decomposing, in order to find  the irriducible components of $\mathcal{P}^{n,k}_{2}$. 
By Lemma \ref{zione}, if\, $k=2$\, then\, $\mathcal{Y}_0\simeq \AA^{\frac{n(n-1)}{2}}$,\, and 
 if $k=3$,  then\, $\mathcal{Y}_0\simeq\mathcal{P}^{n,1}_{2}\times \AA^{\frac{n(n-1)}{2}}$\, which is irreducible since\, $\mathcal{P}^{n,1}_{2}$\, is a classical pfaffian variety - see  Proposition \ref{classpfaff}; thus, if $k=2, 3$, then $\mathcal{Y}_0$ is irreducible.  \newline
Let now  $k\geq 4$, so that\, $\mathcal{Y}_0\simeq \mathcal{P}^{n,k-2}_{2}\times \AA^{\frac{n(n-1)}{2}}$;\, let   $\mathcal{Y}_1$ be the subvariety of $\mathcal{Y}_0$ of the points with\, $x_{ij}^{(1)}=0$\, for all $i,j$; 
we let 
\,$\mathcal{Z}_1=\mathcal{T}_1\times \AA^{\frac{n(n-1)}{2}}$,\, where $\mathcal{T}_1$ is isomorphic to the subvariety of $\mathcal{P}^{n,k-2}_{2}$ which is the closure of the open set where some\, $x_{ij}^{(1)}\not= 0$.\, Roughly speaking, $\mathcal{T}_1$ is the ``$\mathcal{Z}$'' subvariety of $\mathcal{P}^{n,k-2}_{2}$. By Lemma \ref{zione},
$\mathcal{Y}_1\simeq \left\{ \begin{array}{lcl}
  \AA^{n(n-1)} &\text{ if } &k=4\\
  \mathcal{P}_2^{n, k-4} \times \AA^{n(n-1)} &\text{ if } &k>4
\end{array}
\right.$
\noindent
and, if  $k=4$, then $\mathcal{Y}_1$ is irreducible. 

\noindent
By repeating this procedure $\ell-1$ times, where $\ell=\lfloor k/2 \rfloor$, we obtain subvarieties  $\mathcal{Y}_s$, for  $s=1,\ldots,\ell-1$,  defined by $x_{ij}^{(h)}=0$ for $h=0,\ldots, s$ and for all $i,j$. For\, $s\leq l-2$,\, we have \, $\mathcal{Y}_s\simeq         \mathcal{P}^{n,k-(2s+2)}_2\times\AA^{(s+1)\frac{n(n-1)}{2}} $\,
and \, $\mathcal{Z}_{s+1}=\mathcal{T}_{s+1}\times\AA^{(s+1)\frac{n(n-1)}{2}}$,\, where $\mathcal{T}_{s+1}$ is isomorphic to the subvariety of $\mathcal{P}^{n,k-(2s+2)}_2$ which is the closure of the open set where some\, $x_{ij}^{(s+1)}\not= 0$.\, Moreover, $\mathcal{Y}_{\ell-1}\simeq \left\{\begin{array}{lcl}
\AA^{\frac{\ell n(n-1)}{2}} &\text{ if } &k=2\ell;\\
\mathcal{P}^{n,1}_{2}\times \AA^{\frac{\ell n(n-1)}{2}} &\text{ if } &k=2\ell+1.
\end{array}
\right.$
This descends by the very definition of the jet equations given by $4$-pfaffians.
In fact, if $k=2\ell$, the defining equations of $\mathcal{Y}_{\ell-1}$ simply become  $x_{ij}^{(h)}=0$ for $h=0,\ldots,\ell-1$, and  the remaining variables are free;  if $k=2\ell+1$, the equations defining $\mathcal{Y}_{\ell-1}$ do not involve the variables $x_{ij}^{(h)}$ for $h=\ell+1,\ldots,k$, which are therefore free in the corresponding coordinate ring as in the previous case. Furthermore,  $\mathcal{Y}_{\ell-1}$ is defined by the equations of the classical pfaffian variety $\mathcal{P}^{n,1}_2$ expressed in terms of the variables $x_{ij}^{(l)}$. 

\begin{remark}\label{ypsilonesse}
  With the above notation, we can immediately generalize Lemma \ref{new}: for\, $s=0,\ldots,\ell-1$,\, one has $\mathcal{Y}_s\not\subseteq \mathcal{Z}_0$. To this end, one only needs to observe that the point $\mathscr{P}$ defined in \eqref{crux2} belongs to $\mathcal{Y}_s\setminus \mathcal{Z}_0$,\, for all $s=0,\ldots,\ell-1$.
\end{remark}

\noindent
As a first consequence of the above discussion, we obtain the following result.

\begin{proposition}\label{mainred}
  Let $\ell=\lfloor k/2\rfloor$. With the above notation,
  \begin{itemize}    
  \item[a)]  For all \,$s=0,\ldots,\ell-1$,\, the subvarieties\, $\mathcal{Z}_s\subseteq \mathcal{P}^{n,k}_2$\,  are irreducible of codimension\, $\frac{(k-2s)(n-2)(n-3)}{2}+s\frac{n(n-1)}{2}$.
    \medskip
  \item[b)] the subvariety $\mathcal{Y}_{\ell-1}$ is irreducible with
    $$\codim\mathcal{Y}_{\ell-1}=\left\{\begin{array}{lcl}
  \frac{\ell n(n-1)}{2}, &\text{ if }& k=2\ell;\\
  &&\\
  \frac{\ell n(n-1)}{2}+\frac{(n-2)(n-3)}{2}, &\text{ if }& k=2\ell+1.
  \end{array}\right.$$
\end{itemize}
\end{proposition}
\begin{proof} 
  Recall that, from the above discussion,  for all\, $s=0,\ldots,l-1$,\, one has\, $\mathcal{Z}_{s}=\mathcal{T}_{s}\times \AA^{s\frac{n(n-1)}{2}}$,\, where $\mathcal{T}_{s}$ is isomorphic to the subvariety of $\mathcal{P}^{n,k-2s}_{2}$ which is the closure of the open set where some\,  $x_{ij}^{(s)}\not= 0$.\, It follows from Theorem \ref{redusion} that the irreducible components of $\mathcal{T}_{s}$ are in correspondence with the components of\, $\mathcal{P}^{n-2,k-2s}_{1}$,\, and this correspondence preserves the codimensions. Thus, since  $\mathcal{P}^{n-2,k-2s}_{1}$ is irreducible of codimension\, $\frac{(k-2s)(n-2)(n-3)}{2}$,\, then also $\mathcal{Z}_{s}$ is irreducible and $\codim \mathcal{Z}_{s}=\frac{(k-2s)(n-2)(n-3)}{2}+s\frac{n(n-1)}{2}$,\, and this proves the first part of the statement. \newline
  For Part {\em b}), we have already determined $\mathcal{Y}_{l-1}$ and it is immediate to check that it is irreducible. Moreover, it is easy to compute its codimension: if $k=2\ell$, then\,  ${\rm codim\,} \mathcal{Y}_{\ell-1}=\frac{\ell n(n-1)}{2}$;\,  by  Proposition \ref{classpfaff}, the codimension of\, $\mathcal{P}^{n,1}_{2}$\, is\,  $\frac{n(n-1)}{2}-2n+3$\, and therefore, if\, $k=2\ell+1$,\, the codimension of  $\mathcal{Y}_{\ell-1}$ is 
\,$\frac{n(n-1)}{2}-2n+3+\frac{\ell n(n-1)}{2}\,=\,\frac{(n-2)(n-3)}{2}+\frac{\ell n(n-1)}{2}$.

\end{proof}

\begin{lemma}\label{cista?}
  Let $r=2$ and $n\geq 6$.
  \begin{itemize}
  \item[a)] If $n=6, 7$,\, then ${\rm codim\,} \mathcal{Z}_0< {\rm codim\,} \mathcal{Z}_1< \ldots < {\rm codim\,} \mathcal{Z}_{\ell-1}<{\rm codim\,} \mathcal{Y}_{\ell-1}$. \newline
    Moreover,
    $\mathcal{Z}_0$ is the irreducible component of $\mathcal{P}_2^{n,k}$ of smallest codimension and
    $\codim\mathcal{P}^{n,k}_2=\codim\mathcal{Z}_0$.
    \medskip
  \item[b)] If $n\geq 8$, then $\codim \mathcal{Y}_{\ell-1} <\codim \mathcal{Z}_{\ell-1}<\ldots<\codim \mathcal{Z}_{1}<\codim \mathcal{Z}_{0}$. \newline
    Moreover, $\mathcal{Y}_{\ell-1}$ is the irreducible component of $\mathcal{P}_2^{n,k}$ of smallest codimension and
    ${\rm codim\,}\mathcal{P}^{n,k}_2={\rm codim\,}\mathcal{Y}_{\ell-1}$.  
    
  \end{itemize}
\end{lemma}
 \begin{proof} 
   It is clear by definition that, for every $s$, $\mathcal{Z}_s$ cannot be contained in $\mathcal{Z}_{s+1}$, since $x_{ij}^{(s)}=0$ in $\mathcal{Z}_{s+1}$ but not in $\mathcal{Z}_s$; now, a direct computation shows that the codimension of $\mathcal{Z}_s$ is a strictly decreasing function in $s$ if and only if $n^2-9n+12>0$, i.e for all $n\geq 8$; it is strictly increasing for $n=6,7$. 

%

\end{proof}
 
 \begin{proposition}\label{per2decomponiamo}
Let $n\geq 8$ and, for any integer $k\geq 2$, let $\ell=\lfloor k/2\rfloor$. Then, the decomposition of $\mathcal{P}^{n,k}_2$ into irreducible components is given by\,\,  $\mathcal{P}^{n,k}_2=\mathcal{Y}_{\ell-1}\cup \mathcal{Z}_0\cup \mathcal{Z}_1\cup\cdots\cup \mathcal{Z}_{\ell-1}$.
\end{proposition}

  \begin{proof}
We have already observed that $\mathcal{Z}_s$ cannot be contained in $\mathcal{Z}_{s+1}$. Since $n\geq 8$, by Lemma \ref{cista?}, $\mathcal{Z}_s$ cannot contain $\mathcal{Z}_{s+1}$. In a similar fashion we may conclude that $\mathcal{Z}_s$ cannot be contained in $\mathcal{Y}_{\ell-1}$ and cannot contain $\mathcal{Y}_{\ell-1}$; now  the conclusion is straightforward.     
\end{proof}

\begin{corollary}\label{perpiudi2}
Let $n\geq 2r+4$\, and \, $r\geq 2$.\, Then  $\mathcal{P}^{n,k}_{r}$ has at least $\lfloor k/2 \rfloor+1$ irreducible components.   
\end{corollary}
\begin{proof} We argue by induction on $r$.
If $r=2$ then $n\ge 8$ and  the result holds true by Proposition \ref{per2decomponiamo}. Let now $r>2$; since  $n\ge 2r+4$, then  $n-2\ge 2(r-1)+4$; thus by induction $\mathcal{P}^{n-2,k}_{r-1}$ has  at least $\lfloor k/2 \rfloor+1$ irreducible components, and Theorem \ref{redusion} yields that these are in bijection with the irreducible components of   the  subvariety $\mathcal{Z}_0$ of $\mathcal{P}^{n,k}_{r}$; this is enough to imply what desired.
\end{proof}

\begin{remark}\label{docampo}
  The above result is analogous to that of \cite[Theorem 6.1]{KoSe}, which treats the general situation of non-maximal minors in the generic case. A more
  precise statement about the number of irreducible components of jet schemes of determinantal varieties is proved in \cite[Theorem A/Corollary 4.13]{Do}.
  At the moment, we cannot provide a counterpart to the latter in the case of pfaffians using these techniques.
\end{remark}

\begin{proof}[Proof of Theorem \ref{reducibell}]
We argue again by induction on $r$, the case $r=2$ being settled  by Proposition \ref{4reducible}. By induction, $\mathcal{P}^{n-2,k}_{r-1}$ is reducible and, again by Theorem \ref{redusion}, so is $\mathcal{P}^{n,k}_{r}$. \newline
It is easy to check by means of Lemma \ref{cista?},  that the irreducible components of $\mathcal{P}^{n,k}_{2}$ have different codimensions; hence, $\mathcal{P}^{n,k}_2$ is not pure, and thus not Cohen Macaulay. The same conclusion holds for   $\mathcal{P}^{n,k}_{r}$, by arguing on $\mathcal{Z}_0$ and applying Theorem \ref{redusion}.

\end{proof}

\begin{example} Conjecture \ref{quno} is false and Question \ref{qdue} has a negative answer for pfaffian jet ideal $I^{n,k}_r$ with $n\ge 2r+2$,  as the following example shows.
  Consider $I^{6,2}_2$, the jet pfaffian ideal generated by all of the 4 pfaffians of a generic $6\times 6$ skew symmetric matrix $X(t)$ of polynomials of degree 1 in $t$;\, $I^{6,2}_2$ is generated by 30 polynomials of degree $r=2$. With Macaulay2 one verifies  that $\codim I^{6,2}_2=12$, and that the ideal is radical, but not prime.  Moreover, the Hilbert series of $R^{6,2}_2$ is
{\footnotesize  $$H(R^{6,2}_2, z)=\frac{1+12z+48z^2+75z^3+45z^4+15z^5}{(1-z)^{18}},$$}
whereas
{\footnotesize $$H(R^{6,1}_2, z)=\frac{1+6z+6z^2+z^3}{(1-z)^9}.$$}
  Thus, $$\dim R^{6,2}_2=18=2\dim R^{6,1}_2,\,\,\text{ and }\,\,e(R^{6,2}_2)=196=14^2=e(R^{6,1}_2)^2$$
  but $h(R^{6,2}_2)\neq(h(R^{6,1}_2))^2$;\,\,therefore\,\, $H(R^{6,2}_2, z) \neq  H(R^{6,1}_2, z)^2.$
\end{example}

\begin{remark}\label{ghorpade}
  Observe that, in the above example, the non trivial component $\mathcal{Z}_0$ of $\mathcal{P}_2^{6,2}$ is defined by the ideal $I_2^{6,2} \,:\, (x_{56}^{(0)})^\infty$; its Hilbert series is
{\footnotesize  $$\frac{1+12z+48z^2+74z^3+48z^4+12z^5+z^6}{(1-z)^{18}}.$$}
  As in the analogous case for the principal component of the 2nd jet ideal of the ideal of $2\times 2$ minors in a generic matrix (cf. the main result of \cite{GhJoSe}, Theorem 18),  we have that its Hilbert series is the square of the original one. We also notice that its $h$-vector has only positive coefficients and it is symmetric.
\end{remark}


\end{document}